\documentclass[11pt]{amsart}

\usepackage{epigamathfr}

\usepackage[french]{babel}

\usepackage{amscd, amssymb, amsmath, mathrsfs, amsthm}
\usepackage[utf8]{inputenc}
\usepackage{tikz-cd}
\usetikzlibrary{decorations.pathreplacing}
\usepackage{booktabs}
\usepackage{color}

\usepackage{multicol}
\usepackage{xspace}
\usepackage{microtype}
\usepackage[all]{xy}

\newtheorem{thm}{Theor\`eme}[section]
\newtheorem{prop}[thm]{Proposition}
\newtheorem{lem}[thm]{Lemme}

\newtheorem{cor}[thm]{Corollaire}

\theoremstyle{definition}
\newtheorem{df}[thm]{D\'efinition}

\newtheorem{rmk}[thm]{Remarque}

\newcommand{\A}{\mathbb{A}}
\newcommand{\QQ}{\mathbb{Q}}
\newcommand{\FF}{\mathbb{F}}
\newcommand{\ZZ}{\mathbb{Z}}

\newcommand{\Z}{\mathbb{Z}}
\newcommand{\HH}{\mathbb{H}}
\newcommand{\WW}{\mathbb{W}}
\newcommand{\VV}{\mathbb{V}}
\newcommand{\UU}{\mathbb{U}}
\newcommand{\OO}{\mathcal{O}}
\newcommand{\s}{\infty}
\newcommand{\NN}{\mathbb{N}}
\renewcommand{\D}{\mathcal{D}}
\newcommand{\holim}{\mathrm{ho}\lim}

\DeclareMathOperator{\Spec}{Spec}
\DeclareMathOperator{\Map}{Map}
\DeclareMathOperator{\Hom}{Hom}
\DeclareMathOperator{\Bin}{Bin}
\DeclareMathOperator{\colim}{colim}
\newcommand{\fpqc}{\ensuremath{\mathrm{fpqc}}\xspace}
\DeclareMathOperator{\op}{op}
\DeclareMathOperator{\tf}{tf}
\DeclareMathOperator{\uni}{uni}

\EpigaVolumeYear{4}{2020}
\EpigaArticleNr{14}
\ReceivedOn{January 30, 2020}
\InFinalFormOn{July 16, 2020}
\AcceptedOn{August 8, 2020}

\title{Le problème de la schématisation de Grothendieck revisité}
\titlemark{Le problème de la schématisation de Grothendieck revisité}

\author{Bertrand To\"en}
\address{CNRS, Université de Toulouse, 
Institut de Mathématiques de Toulouse (UMR 5219), 
118, route de Narbonne, 
31062 Toulouse Cedex 9, 
France}
\email{Bertrand.Toen@math.univ-toulouse.fr}

\authormark{B. Toën}

\AbstractInFrench{L'objectif de ce travail est de revenir sur le problème de
la schématisation de \cite{poursuite}, particulièrement dans
le cas global au-dessus de $\Spec(\ZZ)$. Pour cela nous
démontrons la conjecture \cite[Conj.2.3.6]{to} qui exprime
les groupes d'homotopie de l'affinisation d'un type
d'homotopie simplement connexe et de type fini. Nous en déduisons plusieurs résultats sur le comportement
du foncteur d'affinisation au-dessus de $\Spec(\ZZ)$, que nous proposons
comme une solution au problème de la schématisation de Grothendieck.}



\TitleInEnglish{Grothendieck's schematization problem revisited}

\AbstractInEnglish{The objective of this work is to reconsider the schematization problem of \cite{poursuite}, with a particular focus on the global case over $\Spec(\ZZ)$. For this, we prove the conjecture \cite[Conj.2.3.6]{to} which gives a formula for the homotopy groups of the schematization of a simply connected homotopy type. We deduce from this several results on the behaviour of the schematization functor  over $\Spec(\ZZ)$, which we propose as a solution to the schematization problem. }

\acknowledgement{Ce travail a été en partie financé par l'ERC au sein du programme
Horizon 2020 (projet NEDAG ERC-2016-ADG-741501).
}

\begin{document}



\maketitle

\begin{prelims}

\DisplayAbstractInFrench

\bigskip

\DisplayKeyWordsfr

\medskip

\DisplayMSCclassfr

\bigskip

\languagesection{English}

\bigskip

\DisplayTitleInEnglish

\medskip

\DisplayAbstractInEnglish

\end{prelims}


\newpage

\setcounter{tocdepth}{1} 

\tableofcontents


\section*{Introduction}

Dans \cite{poursuite} Grothendieck présente un programme 
pour appréhender de manière algébrique les types d'homotopie:
\emph{la schématisation de la théorie de l'homotopie}. Ce programme,
qui reste inachevé dans le manuscrit, consiste à définir une
notion de \emph{type d'homotopie schématique} au-dessus de
$\ZZ$, ainsi qu'un foncteur de \emph{schématisation}, 
qui à tout espace $X$ associe un type d'homotopie schématique $X\otimes \ZZ$,
et qui permette une description purement algébrique de l'homotopie de $X$. 
Plus généralement, il doit exister une notion relative 
à tout anneau commutatif de base $R$, de sorte à ce que l'on retrouve essentiellement
les théories d'homotopie rationnelles et $p$-adiques pour $R=\QQ$ et 
$R=\FF_p$. Lorsque $R=\ZZ$ la théorie des types d'homotopie schématiques
se doit d'être, d'après Grothendieck, aussi proche que possible de celle des types 
d'homotopie usuels.

Ce programme a été en partie réalisé par plusieurs auteurs, et 
nous renvoyons en particulier à \cite{ek,ma1,to} et aux résultats spectaculaires
de \cite{ma2}. Dans \cite{to} nous avons proposé de réaliser la schématisation
des types d'homotopie à l'aide de la notion de \emph{champs affines}, 
qui sont les champs déterminés par leurs algèbres cosimpliciales de cohomologie. 
Le foncteur de schématisation, au-dessus d'un anneau $R$, 
quand à lui est réalisé par le \emph{foncteur
d'affinisation} $X \mapsto (X\otimes R)^{\uni}$, qui à un espace $X$ associe
un champ affine universel construit sur $X$. Lorsque $R=\QQ$ nous avons 
montré que $(X\otimes \QQ)^{\uni}$ est un modèle au type d'homotopie rationnel
de $X$. De même, lorsque $R=\FF_p$, $(X\otimes \FF_p)^{\uni}$ est un 
modèle à la complétion $p$-adique de $X$. Cependant, la description explicite du 
champ $(X\otimes \Z)^{\uni}$, en particulier de ses groupes d'homotopie, 
était laissée sous forme conjecturale (voir \cite[conj.~2.3.6]{to}). \\

Dans ce travail nous donnons une preuve de la conjecture 
\cite[conj.~2.3.6]{to}, ce qui ouvre la voie à de nombreux résultats
sur le comportement du foncteur $X \mapsto (X\otimes \ZZ)^{\uni}$. Pour
cela, nous rappelons l'existence d'un schéma en groupes 
affine $\HH$ que nous appelons le \emph{groupe additif de Hilbert}
(voir définition \ref{d1}). Cette terminologie est justifiée par
le fait que l'algèbre des fonctions sur $\HH$ est l'algèbre
des polynômes à valeurs entières, dont une $\ZZ$-base est fournie
par les célèbres polynômes de Hilbert $\binom{X}{n}$. Ce schéma 
en groupes $\HH$ doit être pensé comme une version entière de la complétion de Malcev
du groupe discret $\ZZ$, et joue en rôle central dans l'étude de l'affinisation
des types d'homotopie. Notons que les points de $\HH$ à valeurs dans un corps
$k$ forment ou bien le groupe additif $(k,+)$ lorsque $\QQ\subset k$, 
ou bien le groupe des entiers $p$-adiques $\hat{\ZZ}_p$ lorsque $k$ est de caractéristique
$p>0$ (voir la proposition \ref{p1} et son corollaire \ref{cp1}). 
Les résultats principaux de ce travail peuvent alors se résumer de la façon suivante.

\begin{thm}\label{ti}
Soit $X$ un  espace simplement connexe de type fini. 

\begin{enumerate}

\item Les faisceaux d'homotopie de $(X\otimes \ZZ)^{\uni}$ sont 
donnés par
$$\pi_i((X\otimes \ZZ)^{\uni}) \simeq \pi_i(X)\otimes \HH.$$

\item Le morphisme d'adjonction
$X \longrightarrow (X\otimes \ZZ)^{\uni}(\ZZ)$
possède une rétraction fonctorielle en $X$. Cette rétraction 
préserve de plus les morphismes induits sur les groupes
d'homotopie. 

\item L'$\s$-foncteur $X \mapsto (X\otimes \ZZ)^{\uni}$ est fidèle
(i.e. injectif sur les groupes d'homotopie des espaces de morphismes).

\item L'$\s$-foncteur $X \mapsto (X\otimes \ZZ)^{\uni}$ est
injectif sur les classes d'équivalence d'objets.

\end{enumerate}
\end{thm}

Il est intéressant de contempler la formule $(1)$
en parallèle des questions considérées dans \cite{poursuite}. En effet, 
une des questions centrales posées par Grothendieck est de savoir
si les groupes d'homotopie de la schématisation de $X$  
doivent posséder, ou non, des structures de $\OO$-modules
(en tant que faisceaux sur le gros site des schémas affines).
Notre résultat répond par la négative, mais en revanche 
montre que ces groupes d'homotopie possèdent des
structures de $\HH$-modules naturelles. Le schéma
en groupes $\HH$ est en réalité le groupe additif sous-jacent
d'un schéma en anneaux, et est de plus muni d'un morphisme
de faisceaux d'anneaux $\HH \longrightarrow \OO$ canonique. Nous aimons
penser à $\HH$ comme au \emph{vrai faisceau structural}, que nous
proposons d'appeler le \emph{faisceau structural de Hilbert}. 

L'existence de la rétraction $(2)$ est une conséquence d'une
description explicite de l'espace $(X\otimes \ZZ)^{\uni}(\ZZ)$, et icelle
est elle-même conséquence de $(1)$ et de techniques de descente
adélique et fidèlement plate.
Cette description est similaire avec 
le résultat analogue dans le cadre des $E_\s$-algèbres
démontré dans \cite[Theorem~0.2]{ma2}, bien que ces deux 
résultats soient indépendants l'un de l'autre. En particulier, 
nous pensons que le théorème \ref{ti} combiné aux résultats
de \cite{ma2} implique que le morphisme naturel
$$\Map_{CAlg^{\Delta}}(C^*(X,\ZZ),\ZZ) \longrightarrow
\Map_{E_\s}(C^*(X,\ZZ),\ZZ),$$
qui compare espaces de morphismes pour les anneaux commutatifs cosimpliciaux 
et
pour  
les $E_\s$-algèbres, est toujours une équivalence (pour $X$ 
simplement connexe et de type fini). Ceci est tout à fait remarquable, 
étant donné que l'$\s$-foncteur de normalisation, des anneaux
commutatifs cosimpliciaux vers les $E_\s$-algèbres, n'est pas
pleinement fidèle en général. Par ailleurs, il 
nous importe de noter ici que les points
$(3)$ et $(4)$ sont conséquences des résultats de \cite{ma2}, transportés
à l'aide de l'$\s$-foncteur de normalisation, mais les preuves que nous
en donnons sont différentes (et indépendantes) et se basent sur 
les résultats $(1)-(2)$ et des techniques générales de la théorie des
champs affines de \cite{to}. Enfin, une comparaison précise 
avec \cite{ma2} demanderait certains efforts supplémentaires, 
efforts que nous ne ferons pas dans ce travail. Par exemple, 
le point $(3)$ ci-dessus est une conséquence d'un énoncé plus fort, 
à savoir l'existence d'une rétraction sur les espaces de morphismes
(voir le corollaire \ref{c7}). L'existence de cette rétraction 
est aussi démontrée dans \cite{ma2} dans le cadre des
$E_\s$-algèbres, mais il ne semble pas immédiat de montrer que ces deux
rétractions soient compatibles. \\

Pour terminer cette introduction, signalons d'étroites relations
avec les travaux \cite{ek} et \cite{mrt}. Le point $(1)$ 
du théorème \ref{ti}, et la propriété universelle de l'affinisation, 
impliquent en particulier que le morphisme naturel
$$H^*(K(\HH,n),\OO) \longrightarrow H^*(K(\ZZ,n),\ZZ)$$
est un isomorphisme pour tout $n$. Le membre de droite est la cohomologie
usuelle des espaces d'Eilenberg-MacLane, alors que le membre
de gauche se calcule à l'aide d'un complexe explicite 
dont le terme de degré $p$ vaut $B^{\otimes p^n}$, où $B$ est l'anneau
des fonctions sur $\HH$, c'est-à-dire l'anneau des polynômes à valeurs
entières. Dans \cite{ek} l'auteur montre aussi que la cohomologie
de $K(\ZZ,n)$, et plus généralement de tout espace nilpotent
de type fini, se calcule à l'aide de 
\emph{cocycles numériques}, c'est-à-dire
provenant de polynômes à valeurs entières. Il y a fort à parier 
que nos résultats sont ainsi très proches de ceux de \cite{ek}, bien 
que la différence de contextes rende une comparaison précise malaisée.
Par ailleurs, le point $(1)$ de notre théorème \ref{ti} implique
en particulier que $(S^1 \otimes \ZZ)^{\uni}\simeq K(\HH,1)$. Cette
formule est déjà démontrée dans \cite{mrt}, tout au moins
dans un cadre $p$-local. La preuve que nous en donnons (voir
le corollaire \ref{c1}) suit la même stratégie mais demande
une étude plus fine pour traiter le cas global sur $\ZZ$. 
Notons aussi que dans \cite{mrt} le champ $K(\HH,1)$
est l'objet sous-jacent du \emph{cercle filtré}, qui joue un rôle
crucial pour la construction de la fameuse filtration HKR. Cette filtration
est ici induite par la filtration naturelle sur le groupe $\HH$, qui 
sur son anneau de fonctions n'est autre que la filtration induite
par le degré des polynômes. Cela suggère que le cercle filtré
de \cite{mrt} n'est qu'un cas particulier d'une filtration 
qui existe naturellement sur le champ $(X\otimes \ZZ)^{\uni}$
pour tout espace $X$, et dont la filtration induite sur le groupes
d'homotopie serait celle induite par la filtration existante sur
$\HH$ à travers la formule $(1)$ du théorème \ref{ti}. Ce point
particulier n'est pas traité dans ce travail et fera l'objet d'investigations
futures.
\vskip\baselineskip
\noindent\textbf{Remerciements.} Je remercie tout particulièrement Joseph Tapia, pour
de très nombreuses discussions sur les vecteurs de Witt, qui m'ont, au cours
de ces années, convaincus de leur importance pour la question
de la schématisation des types d'homotopie de \cite{poursuite}.
Je remercie chaleureusement Tasos Moulinos et 
Marco Robalo, dont les multiples discussions sur le cercle filtré
de \cite{mrt} ont été sources d'inspiration pour un certain nombre de preuves 
des résultats de ce travail. 

\section{Rappels sur les champs affines}

On note $St_\ZZ$  l'$\s$-catégorie des champs \fpqc
sur le gros site des schémas affines. Pour tout
espace $X$ nous noterons encore $X \in St_\ZZ$ le champ
constant associé.
On rappelle que 
l'$\s$-catégorie des champs affines (au-dessus de $\Spec(\ZZ)$)
est la plus petite sous-$\s$-catégorie pleine de $St_\ZZ$ 
qui contient les champs $K(\mathbb{G}_a,n)$ pour tout $n\geq 0$ et 
qui est stable par limites homotopiques ($\UU$-petites, voir notre appendice \ref{appB}).
Cette $\s$-catégorie sera notée $ChAff \subset St_\ZZ$. 

On dispose d'un $\s$-foncteur de sections globales
$\Gamma=(-)(\ZZ) : St_\ZZ \longrightarrow Top,$
de l'$\s$-catégorie des champs vers celle 
des types d'homotopie qui, à un champ $F$, vu comme
foncteur sur la catégorie des anneaux commutatifs, associe
$F(\ZZ)$. Cet $\s$-foncteur possède un adjoint à gauche
$Top \longrightarrow St_\ZZ$ qui à un espace $X$ associe
le champ constant correspondant que nous noterons encore simplement par $X$.  

Nous pouvons retreindre $\Gamma$ à l'$\s$-catégorie $ChAff$ des champs affines.
\emph{L'$\s$-foncteur d'affinisation} est alors l'adjoint à gauche
du foncteur $\Gamma$
$$(-\otimes \ZZ)^{\uni} : Top \longrightarrow ChAff.$$
L'existence de cet $\s$-foncteur (modulo des questions d'univers que nous
ignorons ici) est démontrée
dans \cite{to}, et on peut voir qu'il est donné explicitement par la formule suivante
$$(X\otimes \ZZ)^{\uni}=\mathbb{R}\Spec(C^*(X,\ZZ)).$$
Ici, $C^*(X,\ZZ)$ désigne l'anneau cosimplicial commutatif de cohomologie
de $X$, qui est concrètement donné par le diagramme cosimplicial
$[n] \mapsto \ZZ^{X_n}$, où $X_n$ est l'ensemble des $n$-simplexes
de $X$. Le spectre $\mathbb{R}\Spec(C^*(X,\ZZ))$ est quand à lui le foncteur
sur les anneaux commutatifs défini par 
$$\mathbb{R}\Spec(C^*(X,\ZZ))(R) = \Map(C^*(X,\ZZ),R),$$
où le Map du membre de droite désigne ici le \emph{mapping space}
de la catégorie de modèles des anneaux cosimpliciaux commutatifs
(ou de manière équivalente les espaces de morphismes
de l'$\s$-catégorie d'iceux). L'ensemble
simplicial $\Map(C^*(X,\ZZ),R)$ se décrit explicitement comme 
$\Hom(Q(C^*(X,\ZZ)),C^*(\Delta^*,R))$, où $Q$ est un foncteur de remplacement cofibrant, 
et $C^*(\Delta^*,R)$ est l'objet simplicial (dans la catégorie des
anneaux commutatifs cosimpliciaux) donné par $[n] \mapsto C^*(\Delta^n,R)$. \\

Nous rappelons enfin les faits suivants
démontrés dans \cite{to}.

\begin{enumerate}

\item La sous-$\s$-catégorie $ChAff \subset St_\ZZ$ est stable
par limites homotopiques arbitraires ($\UU$-petites).

\item Un champ $F\in St_\ZZ$ est affine si et seulement s'il existe
un anneau commutatif cosimplicial $R^*$ et une équivalence
$F \simeq \mathbb{R}\Spec(R^*)$. 

\item Les champs $K(\mathbb{G}_a,n)$ sont affines pour tout $n\geq 0$ et sont  donnés par 
$K(\mathbb{G}_a,n) \simeq \mathbb{R}\Spec(\ZZ[x_n]),$
où $\ZZ[x_n]$ est l'anneau commutatif cosimplicial libre sur un générateur
en degré $n$.

\item Pour tout espace $X$ simplement connexe et 
de type fini (i.e. $\pi_i(X)$ est de type fini pour tout $i>1$) le morphisme
naturel
$$X \longrightarrow (X\otimes \ZZ)^{\uni}(\QQ)$$
identifie $(X\otimes \ZZ)^{\uni}(\QQ)$
avec la rationalisation $X_\QQ$ de $X$. C'est-à-dire que les 
groupes d'homotopie de $(X\otimes \ZZ)^{\uni}(\QQ)$ sont les rationalisés
de ceux de $X$.

\item Pour un espace $X$ simplement connexe et 
de type fini et tout corps algébriquement clos $k$ de caractéristique $p>0$ le morphisme
naturel
$$X \longrightarrow (X\otimes \ZZ)^{\uni}(k)$$
identifie $(X\otimes \ZZ)^{\uni}(k)$
avec la complétion $p$-adique $X_p$ de $X$. C'est-à-dire que les 
groupes d'homotopie de $(X\otimes \ZZ)^{\uni}(k)$ sont les complétés $p$-adiques
de ceux de $X$. Plus précisément le champ 
$(X\otimes k)^{\uni}$ est constant de fibre l'espace $X_p$, complété $p$-adique
de $X$ (voir corollaire \ref{ca2}).

\end{enumerate}

Pour terminer, rappelons le fait suivant, que nous utiliserons à plusieurs
reprises dans la suite de cet article. Il s'agit d'une conséquence de 
\cite[th.~2.2.9]{to} qui affirme qu'une affinisation est 
une $\OO$-localisation (voir la remarque après \cite[d\'efi.~2.3.1]{to}).

\begin{prop}\label{prapp}
Soit $X$ un espace topologique et $F$ un champ affine
muni d'un morphisme $u : X \to F(\ZZ)$ $($ou 
de manière équivalente du champ
constant $X$ vers $F)$. Alors les deux conditions
suivantes sont équivalentes.
\begin{enumerate}
\item Le morphisme déduit de $u$ par adjonction
$$(X\otimes \ZZ)^{\uni} \longrightarrow F$$
est une équivalence.

\item  Le morphisme induit par $u$ en cohomologie
$$u^* : H^*(F,\OO) \longrightarrow H^*(X,\ZZ)$$
est bijective.
\end{enumerate}
\end{prop}

\begin{proof}
Il s'agit de remarquer que la cohomologie
à coefficients dans $\OO$ est représentable par les
champs $K(\mathbb{G}_a,n)$
$$H^{n-i}(F,\OO) \simeq \pi_i(\Map(F,K(\mathbb{G}_a,n))) \qquad
H^{n-i}(X,k) \simeq \pi_i(\Map(X,K(\mathbb{G}_a,n))),$$
et d'utiliser que tout champ affine est une limite
petite de champs de la forme $K(\mathbb{G}_a,n)$.
\end{proof}

\section{Le groupe additif de Hilbert $\HH$}

On note $B$ le sous-anneau de $\mathbb{Q}[X]$ form\'e
des polynômes $P$ tels que $P(\Z) \subset \Z$. L'anneau $B$ est l'anneau
binomial libre sur un générateur, et on renvoie
à \cite{ell} pour la notion générale d'anneaux binomiaux. On rappelle que
$B$ est un $\Z$-module libre de base les polynômes de Hilbert
$${X \choose n} := \frac{X(X-1).\dots.(X-n+1)}{n!}.$$

L'anneau $B$ est une $\Z$-algèbre de Hopf pour la comultiplication 
induite par la comultiplication usuelle sur $\mathbb{Q}[X]$
qui envoie $X$ sur $X\otimes 1 + 1 \otimes X$. L'algèbre de Hopf $B$ n'est autre que 
le dual $\ZZ$-linéaire de l'algèbre de Hopf complète correspondant
au groupe formel $\hat{\mathbb{G}}_m$. 

\begin{df}\label{d1}
Le \emph{groupe additif de Hilbert} $\HH$ est le schéma en groupes sur
$\Spec(\ZZ)$ d\'efini par
$$\HH:=\Spec(B).$$
\end{df}

On note que $\HH$ est par définition un schéma en groupes 
affine et plat au-dessus de $\Spec(\Z)$. Il s'agit du dual 
de Cartier du groupe formel $\hat{\mathbb{G}}_m$ (voir \cite{car}). Ses fibres au-dessus
de corps peuvent être décrites explicitement comme suit.
Soit $k$ un corps de caractéristique $p\geq 0$. Ou bien 
$p=0$ et alors $\HH\otimes_\Z k$ est isomorphe au groupe
additif $\mathbb{G}_{a,k}$ au-dessus de $\Spec(k)$. Ou bien 
$p>0$ et alors $\HH\otimes_\Z k$ est isomorphe au groupe
proétale $\hat{\Z}_p$ des entiers $p$-adiques (voir le corollaire \ref{cp1}
pour une preuve de ce fait). 

Il existe une définition alternative de $\HH$ comme vecteurs de Witt
fixes par tous les endomorphismes de Frobenius. On note, pour tout anneau
commutatif $R$,
$\WW(R)$ l'anneau des gros vecteurs de Witt à coefficients dans $R$
(voir par exemple \cite{haz}). Le
foncteur $R \mapsto \WW(R)$ est représentable par un schéma affine, et 
nous considérerons principalement $\WW$ comme un schéma en groupes affine
(au-dessus de $\Spec(\Z)$) pour la loi additive. On rappelle que le groupe
$\WW(R)$ s'identifie au groupe multiplicatif $1+TR[[T]]$, des séries
formelles à coefficients dans $R$ de termes constants égaux à 1.

Le foncteur $\WW$ est muni, pour chaque entier $n>0$, d'un endomorphisme
de Frobenius $F_n : \WW \longrightarrow \WW$. Cet endomorphisme est caractérisé
par sa fonctorialité en $R$, la formule
$F_n(1-aT)=(1-a^nT)$ pou tout $a\in R$, et sa compatibilité avec la topologie naturelle 
sur $\WW(R)$ (topologie produit).
Les endomorphismes $F_n$ vérifient de plus
la formule $F_nF_m= F_{nm}$, et en particulier commutent entre eux.

\begin{prop}\label{p1}
Il existe un monomorphisme 
$$j : \HH \hookrightarrow \WW$$
qui identifie $\HH$ au sous-foncteur des points fixes
simultanés des endomorphismes $F_n$: pour tout anneau $R$
l'image $j(\HH(R))$ consiste en tous les $x\in \WW(R)$ tels
que $F_n(x)=x$ pour tout entier $n>0$.
\end{prop}

\begin{proof}
L'anneau $\OO(\HH)$ est l'anneau
binomial libre à un générateur $\Bin(\ZZ)$ au sens de \cite{ell}.
On rappelle alors que le foncteur 
d'oubli des anneaux dans les anneaux binomiaux
possède un adjoint à droite, qui à $R$ associe
le sous-anneau $\WW(R)^F$ de $\WW(R)$ formé des points fixes
de tous les Frobenius (voir \cite[Theorem~9.1 (3)]{ell}). Ainsi, on a des bijections fonctorielles
en $R$
$$\HH(R) \simeq \Hom(\Bin(\Z),R) \simeq \Hom_{Bin}(\Bin(\Z),\WW(R)^F) \simeq
\Hom(\ZZ,\WW(R)^F)\simeq \WW(R)^F,$$
où les $\Hom$ désignent ici les ensembles de morphismes 
d'anneaux et $\Hom_{Bin}$ ceux d'anneaux binomiaux.

En termes plus explicite, cela fournit un isomorphisme
de foncteurs en groupes $\HH \simeq \WW^F \subset \WW$
qui peut se décrire de la manière suivante. Un élément
de $\HH(R)$ n'est autre qu'une suite d'éléments
$(a_1,a_2,\dots,a_n,\dots)$ qui vérifient les 
relations binomiales (voir par exemple \cite[Lemma~3.3]{ell}). \`A une telle suite la bijection
précédente associe le vecteur de Witt
$$(1-T)^{a_*}:=\sum_{i}(-1)^ia_iT^i \in \WW(R).$$
\end{proof}

La proposition \ref{p1} permet de préciser la structure du schéma en groupes
$\HH$ localement autour d'un nombre premier fixé $p$. Nous savons, d'après 
\cite[Proposition~6]{hes}, que le schéma en groupes $\WW_{\ZZ_{(p)}}:=\WW\times \Spec(\ZZ_{(p)})$,
restriction de $\WW$ au-dessus de $\Spec(\ZZ_{(p)})$, se décompose en un produit
infini
$$\WW_{\ZZ_{(p)}} \simeq \prod_{p\nmid n}\WW_{p^\s},$$
où $\WW_{p^\s}$ est le schéma en groupes des vecteurs de Witt $p$-typiques. Les
Frobenius $F_n$ opèrent sur ce produit de la manière suivante. Si $n$ est premier
à $p$, alors $F_n$ opère sur une suite d'éléments 
$(a_i)_{p\nmid i}$ par la formule $F_n(a_*)_i:=a_{ni}$. De plus, 
$F_p$ opère de manière diagonale par le $p$-ième Frobenius sur
$\WW_{p^\s}$.

Ainsi, la proposition \ref{p1} montre que le schéma en groupes 
$\HH_{\ZZ_{(p)}}$,
restriction de $\HH$ au-dessus de $\Spec(\ZZ_{(p)})$, est canoniquement isomorphe
au groupe des points fixes de $F_p$ opérant sur $\WW_{p^\s}$. Le schéma en groupes
$\HH_{\ZZ_{(p)}}$ est ainsi isomorphe au schéma en groupes
considéré dans \cite{mrt}, et noté $Fix$.

\begin{cor}\label{cp1}
Pour tout premier $p$ et tout entier $n$, soit $\HH_{\ZZ/p^n}$ la restriction
de $\HH$ au-dessus de $\Spec(\ZZ/p^n)$. On dispose d'un isomorphisme de faisceaux \fpqc au-dessus de $\Spec(\ZZ/p^n)$
$$\HH_{\ZZ/p^n} \simeq \hat{\ZZ}_p=\lim_k \ZZ/p^k,$$
où le membre de droite est le faisceau \fpqc constant associé
au groupe profini $\hat{\ZZ}_p$ des entiers $p$-adiques.
\end{cor}

\begin{proof}
On sait qu'il existe une suite exacte
$$\xymatrix{
0 \ar[r] & \HH_{\ZZ/p^n} \ar[r] & \WW_{p^\s} \ar[r]^-{Id-F_p} & \WW_{p^\s}}$$
Le morphisme $Id-F_p$ est un morphisme formellement étale de schémas affines
et plats sur $\ZZ/p^n$, car cela peut se tester au-dessus 
du corps résiduel $\FF_p$ pour lequel la dérivée de $Id-F_p$ est 
partout l'identité. On a donc que $\HH_{\ZZ/p^n}$ est un schéma
en groupes affine plat et formellement étale au-dessus 
de $\ZZ/p^n$. Or, le foncteur 
de changement de bases le long de $\ZZ/p^n \longrightarrow \FF_p$
induit une équivalence de catégories entre les catégories des schémas affines
plats et formellement étales sur $\ZZ/p^n$ et sur $\FF_p$. Ainsi, 
comme $\HH_{\FF_p}$ est constant de fibre $\hat{\ZZ}_p$, il en est de même
de $\HH_{\ZZ/p^n}$.
\end{proof}

\section{L'affinisation de $K(\Z,n)$}

Nous avons vu que $\HH$ s'identifie au sous-groupe $\WW^F \subset \WW$
des points fixes simultanés des Frobenius $F_n$. Ceci va  maintenant nous
permettre de démontrer la proposition suivante.

\begin{prop}\label{p2}
Le champ $K(\HH,n)$ est un champ affine au sens de \cite{to}.
\end{prop}

\begin{proof}
On ordonne les nombres premiers
$2=p_2 < p_3 < p_4 \dots < p_k < p_{k+1} \dots$.
Pour un entier $k\geq 2$ le sous-foncteur
en groupes $\HH_k \subset \WW$ est alors défini comme les points fixes
de tous les Frobenius $F_{p_{i}}$ avec $i\leq k$. On dispose ainsi d'une
suite décroissante de sous-foncteurs en groupes
$$\HH \subset \dots \subset \HH_k \subset \HH_{k-1} \subset \dots \subset
\HH_2 \subset \HH_1:=\WW.$$
Il s'agit d'une suite de sous-schémas en groupes affines, fermés
dans $\WW$, et par définition nous avons
un isomorphisme de faisceaux \fpqc
$$\HH \simeq \lim_k{\HH_k}.$$

Nous utiliserons à plusieurs reprises le résultat classique 
suivant, pour lequel
nous ne connaissons pas de référence et dont nous incluons donc une
preuve. Pour des schémas de type fini il s'agit
d'un cas particulier du critère de platitude par fibres 
\cite[cor.~11.3.11]{ega4-3}.

\begin{lem}\label{l2}
Soit $f : X \longrightarrow Y$ un morphisme de schémas affines et plats sur $\ZZ$. 
Si pour tout corps $k$ le morphisme induit par
changement de bases
$$X\times \Spec(k) \longrightarrow Y\times \Spec(k)$$
est plat (resp. fidèlement plat), alors $f$ est plat (resp. fidèlement plat).
\end{lem}

\begin{proof}
Notons $X\times \Spec(k)$ par $X_k$, et 
de même $Y \times \Spec(k)$ par $Y_k$ (pour tout corps $k$). 
De même notons $X=\Spec(B)$ et $Y=\Spec(A)$.

Soit $A \rightarrow B$ le morphisme d'anneaux correspondant
à $f$ et $M$ un $A$-module et commençons par l'énoncé 
de platitude. Il nous faut donc montrer que le complexe $N:=M\otimes_A^{\mathbb{L}}B$ 
est cohomologiquement concentré en degré zéro. Pour cela, on prend tout d'abord
$k=\mathbb{Q}$, et on voit que $N\otimes_{\ZZ}\QQ$ est 
cohomologiquement concentré en degré $0$. 
Ainsi, les groupes de cohomologie $H^i(N)\simeq H^i(M)\otimes \QQ$ 
sont tous nuls pour $i<0$, et ainsi $H^i(M)$ est de torsion pour $i<0$.

De plus, l'hypothèse implique que $X_{\FF_p}$ est plat sur $Y_{\FF_p}$, et la formule
du changement de base montre alors que $N\otimes_{\ZZ}^{\mathbb{L}}\ZZ/p$ est cohomologiquement concentré
en degrés $[-1,0]$ pour tout premier $p$. Ceci montre tout d'abord 
que $N$ est lui-même cohomologiquement 
concentré en degrés $[-1,0]$. Par ailleurs, comme complexe
de groupe abélien on a $N\simeq H^{-1}(N)[1]\oplus H^0(N)$ avec 
$H^{-1}(N)$ de torsion.
Supposons $H^{-1}(N)\neq 0$. Alors il possède
un élément non-nul et de $p$-torsion pour un certain premier $p$. Comme 
$H^{-2}(N\otimes_{\ZZ}^{\mathbb{L}}\ZZ/p)$ contient 
$Tor_1(H^{-1}(N),\ZZ/p)$ en facteur direct il est donc lui aussi non-nul,
ce qui est une contradiction. 
Ainsi $H^{-1}(N)=0$ et on a donc bien que $A \rightarrow B$ est plat. 
Enfin, pour l'énoncé de fidèle platitude, on sait déjà que $f$ est plat
et il reste donc à voir la surjectivité de $f$ sur les points à valeurs dans des corps. Mais
ceci est clairement une conséquence de l'hypothèse. 
\end{proof}

Pour tout $k$, on dispose d'une suite exacte de faisceaux abéliens
$$\xymatrix{ 0\ar[r] & 
\HH_k \ar[r] & \HH_{k-1} \ar^-{Id-F_{p_k}}[r] & \HH_{k-1}.}$$
La décomposition en produit infini $\WW_{\FF_p} \simeq \prod_{p\nmid n}\WW_{p^\s}$
comme schéma en groupes au-dessus de $\Spec(\FF_p)$, et la description
de Frobenius sur ce produit (voir la discussion
avant le corollaire \ref{cp1}) montrent que la suite exacte ci-dessus devient aussi
exacte à droite lorsque restreinte à $\FF_p$. 
Une application du lemme 
\ref{l2} montre que pour tout $k$ le morphisme $Id-F_{p_k}$ est fidèlement plat
et l'on dispose ainsi d'une suite exacte courte de faisceaux \fpqc
$$\xymatrix{ 0\ar[r] & 
\HH_k \ar[r] & \HH_{k-1} \ar^-{Id-F_{p_k}}[r] & \HH_{k-1} \ar[r] & 0.}$$
Par récurrence sur $k$ cela montre aussi que tous les $\HH_k$ sont
des schémas en groupes affines et plats sur $\Spec(\ZZ)$.

\begin{lem}\label{l1}
Pour tout $n \geq 1$, le morphisme naturel
$$K(\HH,n) \longrightarrow \holim_k K(\HH_k,n)$$
est un équivalence de champs.
\end{lem}

\begin{proof}
Commençons par 
remarquer que le foncteur de faisceautisation
pour la topologie \fpqc commute avec tout type 
de limites dénombrables. Il
commute aussi avec les foncteurs dérivés $\lim^i$ pour les diagrammes
dénombrables (voir l'appendice \ref{appA} pour ces deux assertions).
En particulier, pour
une tour de morphismes entre faisceaux en groupes abéliens
$$\xymatrix{\dots \ar[r]  & F_k \ar[r] &  F_{k-1} \ar[r] &  \dots \ar[r] & F_1}$$
les morphismes naturels
$$a(\lim_k F_k) \longrightarrow \lim_kF_k \qquad
a({\lim_k}^iF_k) \longrightarrow {\lim_{k}}^iF_k$$
sont des isomorphismes de faisceaux (où le $\lim^i$ de droite
est calculé dans la catégorie abélienne des faisceaux \fpqc abéliens
et ceux du membre de gauche dans la catégorie des préfaisceaux abéliens). 
En particulier, $(\lim^i_kF_k)\simeq \lim_{k}^i F_k=0$ pour tout $i>1$.
Les
faisceaux d'homotopie non-nuls du champ $\holim_k K(\HH_k,n)$
sont donc ainsi concentrés en degrés $n$ et $n-1$ et sont donnés par
$$\pi_n(\holim_k K(\HH_k,n)) \simeq \HH \qquad
\pi_{n-1}(\holim_k K(\HH_k,n))\simeq {\lim_k}^1(\HH_k).$$
Le morphisme dont il est question dans l'énoncé du 
lemme induit  un isomorphisme
sur les faisceaux $\pi_n$, et il nous reste donc à voir que 
$\lim_k^1\HH_k \simeq 0$.

Soit $\NN^+$ la catégorie des entiers naturels
strictement positifs ordonnés par l'ordre inverse 
de l'ordre usuel. On dispose de trois $\NN^+$-diagrammes
de faisceaux \fpqc abéliens, à savoir 
$k \mapsto \HH_k$, $k\mapsto \WW/\HH_k$ et 
le diagramme constant égal \`a $\WW$. Notons respectivement
par $\HH_*$, $\WW/\HH_*$ et $\WW$ ces trois $\NN^+$-diagrammes, de sorte à 
ce 
qu'ils
soient reliés 
par une suite exacte courte
$$\xymatrix{0 \ar[r] & \HH_* \ar[r] &  \WW \ar[r] & \WW/\HH_* \ar[r] & 0}.$$
La suite exacte longue sur les foncteurs dérivés de la limite
le long de $\NN^+$ nous donne alors une suite exacte de faisceaux abéliens
$$\xymatrix{
\HH=\lim_k \HH_k \ar[r] & \WW \ar[r] & \lim_k (\WW/\HH_k) \ar[r] & \lim_k^1\HH_k \ar[r] & 0.}$$
Rappelons que l'on a une suite exacte courte
$$\xymatrix{
 0\ar[r] & 
\HH_k \ar[r] & \HH_{k-1} \ar^-{Id-F_{p_k}}[r] & \HH_{k-1} \ar[r] & 0,}$$
et donc des isomorphismes $\HH_{k-1}/\HH_k \simeq \HH_{k-1}$. 
De plus, le morphisme $Id-F_{p_k}$ est fidèlement plat.
Ainsi, 
les faisceaux $\HH_{k-1}/\HH_k$ sont tous représentables par des schémas
affines et plats, ce qui implique par récurrence sur $k$ que 
$\WW/\HH_k$ est, lui aussi, représentable par un schéma affine et plat.
Le morphisme
induit sur la limite $\WW \longrightarrow \lim_k(\WW/\HH_k)$
est ainsi un morphisme de schémas en groupes affines et plats sur $\Spec(\ZZ)$. 

Nous appliquons maintenant le lemme \ref{l2} au morphisme
$\WW \longrightarrow \lim_k(\WW/\HH_k)$. Ce morphisme induit 
au-dessus de chaque corps un morphisme fidèlement plat 
de schémas en groupes. En effet, au-dessus d'un corps $k$ le
schéma en groupes $\WW$ se décompose en un produit infini
qui se décrit suivant la caractéristique de $k$ (voir
le paragraphe avant le corollaire \ref{cp1}). 
Il est aisé de voir, à l'aide de ces deux descriptions, que, pour tout corps $k$, le morphisme
de schémas en groupes
$$\WW \times \Spec(k) \longrightarrow \lim_k(\WW/\HH_k) \times \Spec(k)$$
est un épimorphisme (pour la topologie \fpqc) de schémas en groupes affines, 
et donc est un morphisme fidèlement plat de schémas. Le lemme \ref{l2}
implique ainsi que $\WW \longrightarrow \lim_k(\WW/\HH_k)$
est un morphisme fidèlement plat de schémas en groupes affines sur $\Spec(\ZZ)$, 
et donc un épimorphisme de faisceaux \fpqc. Ainsi, nous avons un isomorphisme de faisceaux
$\WW/\HH \simeq \lim_k(\WW/\HH_k)$, ce qui implique que $\lim^1_k(\HH_k)=0$ comme
convenu.
\end{proof}

Nous revenons à la preuve de la proposition \ref{p2}. 
Nous avons vu
l'existence d'une suite de fibrations
de champs affines
$$K(\HH_k,n) \longrightarrow K(\HH_{k-1},n) \longrightarrow K(\HH_{k-1},n).$$
Comme les champs affines sont stables par limites homotopiques (voir 
\cite{to}) on voit
par récurrence sur $k$, que $K(\HH_k,n)$ est un champ
affine pour tout $k$. Enfin, 
$K(\HH,n)$ étant la limite homotopique des $K(\HH_k,n)$, on en déduit
que $K(\HH,n)$ est un champ affine à l'aide de \cite[prop.~2.2.7]{to}, ce qui achève la preuve de la proposition \ref{p2}.
\end{proof}

La proposition \ref{p2} possède la conséquence importante suivante, 
qui répond à la conjecture \cite[conj.~2.3.6]{to}. Notons
que $\WW$ reçoit un unique morphisme additif
$\ZZ \longrightarrow \WW$ donné par l'unité dans l'anneau $\WW(\ZZ)$. Cette
unité appartient à $\HH$ et on a donc ainsi un morphisme 
canonique de faisceaux en groupes $\ZZ \longrightarrow \HH$. Par
construction, ce morphisme envoie un entier $n$ sur la série
$(1-T)^n \in \HH(\ZZ)$ et fournit un isomorphisme sur les sections globales
$\ZZ \simeq \HH(\ZZ)$.

\begin{cor}\label{c1}
Soit $n\geq 1$.
Le morphisme canonique $\mathbb{Z} \longrightarrow \HH$
induit un morphisme de champs
$$K(\ZZ,n) \longrightarrow K(\HH,n)$$
qui fait de $K(\HH,n)$ l'affinisation de $K(\ZZ,n)$ 
au-dessus de $\Spec(\ZZ)$. En d'autres termes, le morphisme
induit en cohomologie
$H^*(K(\HH,n),\OO) \longrightarrow H^*(K(\ZZ,n),\ZZ)$
est un isomorphisme.
\end{cor}

\begin{proof}
Nous savons déjà que 
$K(\HH,n)$ est un champ affine, et il nous faut donc montrer que le morphisme
induit sur les complexes de cohomologie
$$C^*(K(\ZZ,n),\ZZ) \longrightarrow C^*(K(\HH,n),\OO)$$
est un quasi-isomorphisme (voir la proposition \ref{prapp}).

Commençons par le cas $n=1$. Dans ce cas, la formation du complexe
$C^*(K(\HH,1),\OO)$ est compatible aux changements de bases le long
des morphismes $\ZZ \longrightarrow \FF_p$ et $\ZZ \longrightarrow \QQ$. 
En effet, comme $K(\HH,1)$ est le quotient de $\Spec(\ZZ)$ par 
$\HH$, le complexe $C^*(K(\HH,1),\OO)$ peut se représenter comme le 
complexe de cohomologie de Hochschild
$$\xymatrix{\ZZ \ar[r] &  B \ar[r] & B^{\otimes 2} \ar[r] & \dots \ar[r]& B^{\otimes n} \ar[r] & B^{\otimes n+1} \ar[r] & \dots,}$$
où $B=\OO(\HH)$ est l'algèbre de Hopf des polynômes à valeurs entières, et où la différentielle
est induite par la comultiplication $B \longrightarrow B^{\otimes 2}$ 
(voir \cite[\S III]{dg}). Cette présentation, et 
la platitude de $B$ sur $\ZZ$, impliquent que les morphismes induits
$$C^*(K(\HH,1),\OO) \otimes \QQ \longrightarrow C^*(K(\HH_\QQ,1),\OO)$$
$$C^*(K(\HH,1),\OO) \otimes \FF_p \longrightarrow C^*(K(\HH_{\FF_p},1),\OO)$$
sont tous deux des quasi-isomorphismes (où l'on a noté
$\HH_k:=\HH\times \Spec(k)$ pour un corps $k$). 
Ainsi, pour vérifier que le morphisme $C^*(K(\ZZ,1),\ZZ) \longrightarrow C^*(K(\HH,1),\OO)$
est un quasi-isomorphisme, il suffit de montrer l'énoncé analogue pour les
champs $K(\HH_\QQ,1)$ et $K(\HH_{\FF_p},1)$. Or, 
ces champs sont respectivement $K(\mathbb{G}_{a},1)$ et $K(\hat{\ZZ}_p,1)$. Ainsi, le corollaire
pour $n=1$ se déduit du fait que l'on sait déjà que les morphismes naturels
$$K(\ZZ,1) \longrightarrow K(\mathbb{G}_a,1) \qquad K(\ZZ,1) \longrightarrow K(\hat{\ZZ}_p,1)$$
sont des affinisations au-dessus de $\QQ$ et des corps finis 
$\FF_p$ respectivement (voir \cite{to}). 

Enfin, le cas $n>1$ se réduit au cas $n=1$ de la manière suivante. On remarque que
$K(\HH,n)$ et $K(\ZZ,n)$ sont tous deux les classifiants des champs
en groupes $K(\HH,n-1)$ et $K(\ZZ,n-1)$. Ainsi, 
les complexes de cohomologie $C^*(K(\ZZ,n),\ZZ)$ et 
$C^*(K(\HH,n),\OO)$ s'écrivent comme des limites homotopiques
$$C^*(K(\ZZ,n),\ZZ) \simeq \holim_{k\in \Delta}C^*(K(\ZZ,n-1),\ZZ)^{\otimes k}$$
$$C^*(K(\HH,n),\OO) \simeq \holim_{k\in \Delta}C^*(K(\HH,n-1),\OO)^{\otimes k}.$$
Le morphisme $C^*(K(\ZZ,n),\ZZ) \longrightarrow C^*(K(\HH,n),\OO)$ étant compatible 
à ces décompositions, nous concluons par récurrence sur $n$.
\end{proof}

Le corollaire \ref{c1} possède la conséquence intéressante suivante.

\begin{cor}\label{c1'}
Il existe un isomorphisme d'algèbres de Hopf
$$\OO(\HH) \simeq \mathbb{Z}\otimes^{\mathbb{L}}_{C^*(S^1,\ZZ)}\ZZ$$
$($en particulier le membre de droite est cohomologiquement concentré en 
degré $0)$.
\end{cor}

\begin{proof}
En effet, on a un 
carré homotopiquement cartésien de champs affines
$$\xymatrix{
\HH \ar[r] \ar[d] & \Spec(\ZZ) \ar[d] \\
\Spec(\ZZ) \ar[r] & K(\HH,1).}$$
En en prenant les algèbres cosimpliciales de cohomologie, et
en utilisant le corollaire \ref{c1} on trouve
un carré homotopiquement cocartésien d'algèbres cosimpliciales commutatives
$$\xymatrix{
\OO(\HH)  & \ar[l] \ZZ  \\
\ZZ \ar[u] & C^*(S^1,\ZZ) \ar[u] \ar[l],}$$
ce qu'il fallait démontrer.
\end{proof}

Il est possible de poursuivre dans la direction du corollaire \ref{c1'} pour obtenir une interprétation
tannakienne du schéma en groupes $\HH$. Pour cela on considère
$QCoh(S^1)$ l'$\s$-catégorie des complexes de faisceaux abéliens
sur le cercle $S^1$ dont les faisceaux de cohomologie sont localement constants.
C'est une $\s$-catégorie $\ZZ$-linéaire et symétrique monoïdale. On considère
la sous-$\s$-catégorie symétrique monoïdale $QCoh^{\uni}(S^1)$
engendrée par colimites et décalage par l'unité $\OO$. Cette $\s$-catégorie
peut aussi s'écrire de la forme $\D(A)$, l'$\s$-catégorie
des complexe de $A$-dg-module, où $A=C^*(S^1,\ZZ)$ vue comme
$E_\s$-algèbre. Le foncteur fibre $QCoh(S^1) \longrightarrow QCoh(\Spec(\ZZ))=\D(\ZZ)$
induit un $\s$-foncteur $\ZZ$-linéaire et symétrique monoïdal
$$\omega : QCoh^{\uni}(S^1) \longrightarrow \D(\ZZ).$$
Il s'agit, à l'aide de l'identification $QCoh^{\uni}(S^1)\simeq \D(A)$, du changement
de bases le long de l'augmentation naturelle $A \longrightarrow \ZZ$. 

Pour tout anneau commutatif $R$, on dispose de la composition
$$\omega \otimes R : QCoh^{\uni}(S^1) \longrightarrow \D(\ZZ) \longrightarrow
\D(R),$$
qui est encore un $\s$-foncteur $\ZZ$-linéaire et symétrique monoïdal. On dispose
d'une suite exacte d'espaces de morphismes
$$\xymatrix{
\Map_{QCoh^{\uni}(S^1)}(\D(\ZZ),\D(R)) \ar[r] & 
\Map_{\D(\ZZ)}(\D(\ZZ),\D(R)) \ar[r] & \Map_{\D(\ZZ)}(QCoh^{\uni}(S^1),\D(R))}$$
où la fibre est prise au point $\omega\otimes R$, et 
$\Map_{\mathcal{C}}$ désigne ici l'espace des $\s$-foncteurs
continus symétriques monoïdaux au-dessus d'une $\s$-catégorie
symétrique monoïdale fixée $\mathcal{C}$. Comme le terme du milieu
est clairement contractile, nous trouvons une équivalence fonctorielle
en $R$
$$\mathrm{Aut}^{\otimes}(\omega\otimes R) \simeq \Map_{QCoh^{\uni}(S^1)}(\D(\ZZ),\D(R)).$$
Par ailleurs, l'$\s$-foncteur 
$A \mapsto \D(A)$, depuis les $E_\s$-algèbres vers les
$\s$-catégories monoïdales symétriques présentables, est pleinement fidèle
(voir \cite{HA}). Ainsi, le membre de droite de l'équivalence ci-dessus
se décrit par
$$\Map_{QCoh^{\uni}(S^1)}(\D(\ZZ),\D(R)) \simeq \Map(\ZZ\otimes_{A}^{\mathbb{L}}\ZZ,R)
\simeq \HH(R).$$
En somme, nous avons montré que le foncteur $R \mapsto \mathrm{Aut}^{\otimes}(\omega\otimes R)$, des auto-équivalences symétriques monoïdales de $\omega$, est représentable
par le schéma en groupes $\HH$. Un fait que nous pouvons résumer par: 
\emph{le schéma en groupes $\HH$ est le dual de Tannaka des systèmes locaux unipotents
sur $S^1$}. 

\section{Faisceaux d'homotopie de l'affinisation}

On rappelle qu'un espace (connexe et) simplement connexe $X \in Top$ est  \emph{de type fini}
si les groupes d'homotopie $\pi_i(X)$ sont tous de type fini. L'énoncé suivant est le théorème
central de cet article. 

\begin{thm}\label{p3}
Soit $X$ un espace simplement connexe et de type fini. Alors il existe, 
pour tout $i$, des isomorphismes naturels de faisceaux en groupes
$$\pi_i(X)\otimes_\ZZ \HH \simeq \pi_i((X\otimes \ZZ)^{\uni}).$$
\end{thm}

\begin{proof}
Nous allons construire les isomorphismes
$\pi_i(X)\otimes_\ZZ \HH \simeq \pi_i((X\otimes \ZZ)^{\uni})$
au cours de la preuve. Nous procèderons par une récurrence sur la tour
de Postnikov de $X$
$$\xymatrix{X \ar[r] & \dots \ar[r] & X_n \ar[r] & X_{n-1} \ar[r] & \dots & X_{2} \ar[r] & X_1=*,}$$
et notons $\pi_i=\pi_i(X)$. 
Nous commençons ainsi par traiter le cas des espaces d'Eilenberg-MacLane. Notons que
le théorème implique en particulier, que si $X$ est de plus $n$-tronqué (i.e. $\pi_i(X)=0$ pour $i>n$), 
alors le champ $(X\otimes \ZZ)^{\uni}$ est lui-même $n$-tronqué.

\begin{lem}\label{l3}
Soit $\pi$ un groupe abélien de type fini. Alors, pour tout $n>1$, 
on a une équivalence naturelle de champs
$$(K(\pi,n)\otimes \mathbb{Z})^{\uni} \simeq K(\pi\otimes \HH,n).$$
Cette équivalence est de sorte à ce que le morphisme d'adjonction
$K(\pi,n) \to (K(\pi,n)\otimes \mathbb{Z})^{\uni}$ soit de plus induit
par le morphisme naturel de faisceaux $\pi \to \pi\otimes \HH$.
\end{lem}

\begin{proof}
Lorsque $\pi=\ZZ$, ce lemme est la proposition \ref{p2}. 
Le cas $\pi=\ZZ^m$ s'en déduit aisément par K\"{u}nneth
$$(K(\ZZ,n)^m\otimes \mathbb{Z})^{\uni} \simeq \mathbb{R}\Spec(C^*(K(\ZZ,n),\ZZ)^{\otimes m}) \simeq 
\mathbb{R}\Spec(C^*(K(\ZZ,n),\ZZ))^{\times m}\simeq K(\pi\otimes \HH,n).$$
Dans le cas général, toujours par l'argument de K\"{u}nneth, il nous
reste à traiter le cas $\pi=\ZZ/m$ d'un groupe cyclique. Dans ce cas, on dispose d'un carré homotopiquement cartésien
$$\xymatrix{
K(\pi,n) \ar[r] \ar[d] & \bullet \ar[d] \\
K(\ZZ,n+1) \ar[r]_-{\times m} & K(\ZZ,n+1).}$$
Nous sommes dans les conditions d'applications du théorème d'Eilenberg-Moore 
rappelé dans l'appendice \ref{appA}, et ainsi
le carré précédent induit un carré homotopiquement cocartésien d'algèbres cosimpliciales commutatives
$$\xymatrix{
C^*(K(\ZZ,n+1),\ZZ) \ar[r] \ar[d] & C^*(K(\ZZ,n+1),\ZZ) \ar[d] \\
\ZZ \ar[r] & C^*(K(\pi,n),\ZZ).}$$
En en prenant les spectres correspondants, et en utilisant le cas déjà connu de $\pi=\ZZ$, on trouve un carré homotopiquement cartésien
de champs affines
$$\xymatrix{
(K(\pi,n)\otimes \ZZ)^{\uni} \ar[r] \ar[d] & \bullet \ar[d] \\
K(\HH,n+1) \ar[r]_-{\times m} & K(\HH,n+1).}$$
La suite exacte longue en homotopie et le fait que $\HH(R)$ soit un groupe
sans torsion (car sous-groupe de $\WW(R)$) impliquent que l'on a un isomorphisme naturel
$$\pi_{n}((K(\pi,n)\otimes \ZZ)^{\uni}) \simeq \pi \otimes \HH.$$
Ceci termine la preuve du lemme.
\end{proof} 

Le lemme implique l'énoncé du théorème pour des espaces tronqués $X$, 
par une récurrence simple le long de la tour 
de Postnikov de $X$. En effet, on a pour tout $k\geq 3$ un carré homotopiquement cartésien
$$\xymatrix{
X_k \ar[r] \ar[d] & \bullet \ar[d] \\
X_{k-1} \ar[r] & K(\pi_k,k+1),}$$
avec $\pi_k$ un groupe abélien de type fini. Le théorème de Eilenberg-Moore 
\ref{ta1} s'applique, et 
donne comme précédemment un carré homotopiquement cocartésien 
sur les algèbres cosimpliciales de cohomologie. Par le lemme précédent, le diagramme induit sur leurs spectres
est alors un carré homotopiquement cartésien de champs affines
$$\xymatrix{
(X_k \otimes \ZZ)^{\uni} \ar[r] \ar[d] & \bullet \ar[d] \\
(X_{k-1} \otimes \ZZ)^{\uni} \ar[r] & K(\pi_k \otimes \HH,k+1).}$$
La proposition s'ensuit par récurrence et la suite exacte longue en homotopie. 

Enfin, pour un espace simplement connexe de type fini $X$ général,
on écrit $X=\holim_n X_n$ comme limite homotopique de ses tronqués.
En particulier, on a $C^*(X,\ZZ) \simeq \colim_n C^*(X_,\ZZ)$, 
et on en déduit donc $(X\otimes \ZZ)^{\uni} \simeq \holim_n (X_n\otimes \ZZ)^{\uni}$.
Nous connaissons déjà la proposition pour chacun des $X_n$, et pour conclure
il nous suffit donc de montrer que le morphisme naturel
$$(X\otimes \ZZ)^{\uni} \longrightarrow  (X_n\otimes \ZZ)^{\uni}$$
induit des isomorphismes sur les faisceaux d'homotopie $\pi_i$ dès que
$i\leq n$. Or, les préfaisceaux d'homotopie du champ 
$(X\otimes \ZZ)^{\uni}$ entrent dans des suites exactes courtes
$$\xymatrix{ 0 \ar[r] & 
\pi_i^{pr}((X\otimes \ZZ)^{\uni}) \ar[r] 
& \lim_n \pi_i^{pr}((X_n\otimes \ZZ)^{\uni}) \ar[r]
& \lim^1_n \pi_{i+1}^{pr}((X_n\otimes \ZZ)^{\uni}) \ar[r] & 0.}$$
Comme le foncteur faisceau \fpqc associé commute aux $\lim_n$ et $\lim_n^1$ 
(voir corollaire \ref{ca}), on trouve des suites exactes de faisceaux
$$\xymatrix{ 0 \ar[r] & 
\pi_i((X\otimes \ZZ)^{\uni}) \ar[r] 
& \lim_n \pi_i((X_n\otimes \ZZ)^{\uni}) \ar[r]
& \lim^1_n \pi_{i+1}((X_n\otimes \ZZ)^{\uni}) \ar[r] & 0.}$$
Or, le système $n \mapsto \pi_{i+1}((X_n\otimes \ZZ)^{\uni})$ est constant,
de valeurs $\pi_{i+1}(X)\otimes \HH$,
dès que $n\geq i$, et ainsi le terme $\lim^1$ s'annule à l'aide
d'une nouvelle application du corollaire \ref{ca}. On a donc
$$\pi_i((X\otimes \ZZ)^{\uni}) \simeq \lim_n \pi_i((X_n\otimes \ZZ)^{\uni})
\simeq \pi_i((X_n\otimes \ZZ)^{\uni})$$
dès que $n\geq i$. 
\end{proof}

\begin{rmk}\label{rp3}Les isomorphismes
$\pi_i(X)\otimes_\ZZ \HH \simeq \pi_i((X\otimes \ZZ)^{\uni})$ 
du théorème \ref{p3} sont, par observation, fonctoriels en $X$. En effet, 
ils sont construits de la manière suivante. On dispose, pour $i\geq 0$
fixé, d'un diagramme 
d'espaces
$$\xymatrix{X \ar[r] & X_i & K(\pi_i,i) \ar[l]}$$
où $X_i$ est le $i$-ème tronqué de Postnikov de $X$. On en déduit alors
d'un diagramme de faisceaux abéliens
$$\xymatrix{\pi_i((X\otimes \ZZ)^{\uni}) 
\ar[r]^-{c} & \pi_i((X_i\otimes \ZZ)^{\uni}) & \pi_i((K(\pi_i,i)\otimes \ZZ)^{\uni}) \ar[l]_-{b} \ar[r]^-{a} & \pi_i(K(\pi_i \otimes\HH,i))\simeq
\pi_i\otimes \HH.}$$
Les morphismes $c$ et $b$ sont ici induits par fonctorialité
à partir du diagramme d'espaces ci-dessus. Le morphisme
$a$ est quant à lui construit à partir du morphisme
naturel $K(\pi_i,i) \to K(\pi_i \otimes \HH,i)$,
en observant que le lemme \ref{l3} montre en particulier que
$K(\pi_i \otimes \HH,i)$ est un champ affine. Par propriété 
universelle de l'affinisation le morphisme $K(\pi_i,i) \to K(\pi_i \otimes \HH,i)$ induit ainsi un morphisme 
$(K(\pi_i,i)\otimes \ZZ)^{\uni} \to K(\pi_i \otimes \HH,i)$.
De plus, le morphisme $a$ est un isomorphisme d'après le lemme \ref{l3}, et par
construction est fonctoriel en le groupe $\pi_i$. Enfin, 
nous avons vu au cours de la preuve du théorème \ref{p3} que les morphismes
$b$ et $c$ sont des isomorphismes. L'isomorphisme du théorème \ref{p3}
est alors donné par $c^{-1}ba^{-1}$.
\end{rmk}

Un important corollaire du théorème \ref{p3} 
est le fait suivant. Remarquons pour cela que 
$\HH$ est le groupe additif d'un faisceau en anneaux
commutatifs, car $\HH(R)$ s'identifie au sous-anneau
de $\WW(R)$ formé des vecteurs de Witt fixés par tous les Frobenius.

\begin{cor}\label{c2}
Soit $X$ un espace simplement connexe et de type fini, alors 
pour tout $i$ le faisceau en groupes
$\pi_i((X\otimes \ZZ)^{\uni})$ possède une structure naturelle
de faisceau en $\HH$-modules.
\end{cor}

\begin{rmk}
La structure de $\HH$-module du corollaire précédent 
est par définition celle induite par notre théorème \ref{p3}. Malheureusement,
l'isomorphisme $\pi_i(X)\otimes_\ZZ \HH \simeq \pi_i((X\otimes \ZZ)^{\uni})$
de ce théorème est construit au cours de la preuve (voir
remarque \ref{rp3}). Il serait 
intéressant de savoir démontrer le corollaire \ref{c2} de manière 
indépendante, et ce afin que le morphisme naturel 
$\pi_i(X)\to \pi_i((X\otimes \ZZ)^{\uni})$ induise l'isomorphisme
du théorème \ref{p3}. Nous ne savons cependant pas comment 
démontrer que les faisceaux $\pi_i((X\otimes \ZZ)^{\uni})$ possèdent une 
structure naturelle de $\HH$-modules de manière directe.
\end{rmk}

Un second corollaire, important et que nous utiliserons dans la section
suivante, et le phénomène de rigidité suivant. On note
$A=\WW_{p^\s}(k)$ avec $k$ une extension algébrique de $\FF_p$. 
C'est un anneau de valuation 
discrète noethérien et complet
de corps résiduel $k$. On note
$\pi \in A$ une uniformisante de $A$.

\begin{cor}\label{c22}
Soit $X$ un espace simplement connexe et de type fini. Alors, pour
entier $n$, les morphismes
$$\xymatrix{
(X\otimes \ZZ)^{\uni}(A) \ar[r] & (X\otimes \ZZ)^{\uni}(A/\pi^n)
\ar[r] & (X\otimes \ZZ)^{\uni}(k)}$$
sont des équivalences. En d'autres termes
le champ $(X\otimes \ZZ)^{\uni}$, restreint au petit site profini-étale
de $\Spec(\hat{\ZZ}_p)$ est un champ constant de fibre 
$X_p$ le complété p-adique de l'espace $X$.
\end{cor}

\begin{proof}
Tout d'abord, comme $A=lim_n A/\pi^n$, et 
que cette limite peut être comprise comme une limite homotopique
dans les anneaux cosimpliciaux, on a 
$$(X\otimes \ZZ)^{\uni}(A) \simeq \holim_n (X\otimes \ZZ)^{\uni}(A/\pi^n).$$
Il suffit donc de vérifier que pour tout $n$ le morphisme
$(X\otimes \ZZ)^{\uni}(A/\pi^n) \longrightarrow (X\otimes \ZZ)^{\uni}(k)$
est une équivalence. Mais
cela se déduit du théorème \ref{p3}, 
du corollaire \ref{cp1}, et d'un dévissage de Postnikov standard.
\end{proof}

\section{Propriétés d'injectivité du foncteur d'affinisation}

Pour terminer, nous allons étudier les propriétés de l'$\s$-foncteur
d'affinisation 
$$(-\otimes \ZZ)^{\uni} : Top_{\geq 1}^{\tf} \longrightarrow ChAff,$$
où $Top_{\geq 1}^{\tf} \subset Top$ est la sous-$\s$-catégorie pleine
des espaces simplement connexes et de type fini. 
Nous allons d'abord voir que ce foncteur n'est pas pleinement fidèle, 
en décrivant explicitement l'unité de l'adjonction
$$X \longrightarrow (X\otimes \ZZ)^{\uni}(\ZZ).$$

L'anneau $\ZZ$ entre dans un carré cartésien d'anneaux commutatifs
$$\xymatrix{
\ZZ \ar[r] \ar[d] & \hat{\ZZ}=\prod_p \hat{\ZZ}_p \ar[d] \\
\QQ \ar[r] & \A,}$$
où $\A:=(\prod_p \hat{\ZZ}_p)\otimes \QQ$ est l'anneau des adèles. Ce carré est 
encore homotopiquement cartésien lorsque considéré dans l'$\s$-catégorie
des anneaux cosimpliciaux, car il induit une suite exacte courte
$$\xymatrix{ 0 \ar[r] & 
\ZZ \ar[r] & \QQ \oplus \hat{\ZZ} \ar[r] & \A \ar[r] & 0.}$$
Ainsi, pour tout champ affine $F$ on a une décomposition naturelle
$$F(\ZZ)\simeq F(\QQ) \times_{F(\A)}\prod_p F(\hat{\ZZ}_p).$$
Supposons maintenant que $F=(X\otimes \ZZ)^{\uni}$ soit l'affinisé d'un espace
simplement connexe et de type fini. On sait que $\HH$ devient naturellement isomorphe au groupe
additif sur $\Spec(\QQ)$ et ainsi, pour toute $\QQ$-algèbre $R$, on a
$\HH(R)\simeq R$. Mieux, $\mathbb{G}_a$ ne possédant pas de cohomologie sur
les schémas affines, le théorème \ref{p3} nous donne des isomorphismes de groupes abéliens
$$\pi_i((X\otimes \ZZ)^{\uni}(\QQ))\simeq \pi_i(X)\otimes \QQ \qquad
\pi_i((X\otimes \ZZ)^{\uni}(\A))\simeq \pi_i(X)\otimes \A.$$
Par ailleurs, le corollaire \ref{c22}
implique que le champ
$(X\otimes \ZZ)^{\uni}$, restreint au petit site profini-étale 
de $\Spec(\hat{\ZZ}_p)$, est un champ constant de
fibre $X_p$.
En particulier, on trouve une équivalence canonique pour tout 
premier $p$
$$(X\otimes \ZZ)^{\uni}(\hat{\ZZ}_p) \simeq X_p^{S^1},$$
où $X_p^{S^1}$ est l'espace des lacets libres dans $X_p$. 
On trouve ainsi une décomposition naturelle pour tout $X$
$$(X\otimes \ZZ)^{\uni}(\ZZ) \simeq X_\QQ \times_{X_\A}\prod_p X_p^{S^1},$$
où $X_\QQ$ est le rationalisé de $X$, $X_p$ son complété p-adique et 
$X_\A$ est une notation pour $(X\otimes \ZZ)^{\uni}(\A)$.

La conclusion de cette discussion est l'énoncé suivant, qui calcule les groupes
d'homotopie de l'espace $(X\otimes \ZZ)^{\uni}(\ZZ)$ en termes
de ceux de $X$.

\begin{cor}\label{c4}
Pour $X$ un espace simplement connexe 
et de type fini on a des isomorphismes fonctoriels en $X$
$$\pi_i((X\otimes \ZZ)^{\uni}(\ZZ))\simeq \pi_i(X) \oplus \prod_p \pi_{i+1}(X)^{\wedge}_p,$$
où $M^{\wedge}_p=\lim_n M/p^n$ désigne la complétion $p$-adique d'un groupe
abélien $M$.
\end{cor}

\begin{proof}
En effet, posons $F=(X\otimes \ZZ)^{\uni}$. La décomposition 
$$F(\ZZ)\simeq F(\QQ) \times_{F(\A)}\prod_p F(\hat{\ZZ}_p)$$
induit une suite longue en homotopie de la forme
$$\xymatrix{
\pi_i(F(\ZZ)) \ar[r] & \pi_i(F(\QQ)) \oplus \prod_p \pi_i(F(\hat{\ZZ}_p)) \ar[r] &  \pi_i(F(\A)).}$$
Cette suite s'écrit aussi
$$\xymatrix{
\pi_i(F(\ZZ)) \ar[r] & \pi_i(X)_\QQ \oplus \prod_p \pi_i(X)^{\wedge}_p \oplus 
\prod_p \pi_{i+1}(X)^{\wedge}_p
\ar[r] &  \pi_i(X)\otimes \A,}$$
et donne ainsi lieu à des suites exactes courtes
$$\xymatrix{0 \ar[r]& 
\pi_i(F(\ZZ)) \ar[r] & \pi_i(X)_\QQ \oplus \prod_p \pi_i(X)^{\wedge}_p \oplus 
\prod_p \pi_{i+1}(X)^{\wedge}_p
\ar[r] &  \pi_i(X)\otimes \A \ar[r] & 0.}$$
Comme le morphisme $\prod_p \pi_{i+1}(X)^{\wedge}_p \longrightarrow \pi_i(X)\otimes \A$ est 
trivial
par construction, on en déduit le corollaire.
\end{proof}

Une conséquence intéressante du corollaire précédent est la formule suivante, qui 
donne un calcul explicite des groupes de cohomologie $H_{\fpqc}^i(\Spec(\ZZ,\HH))$.

\begin{cor}\label{c5}
Pour tout groupe abélien de type fini $M$
$$H^0(\Spec(\ZZ),M\otimes \HH)\simeq M, \qquad H^i_{\fpqc}(\Spec(\ZZ),M\otimes \HH)=0
(\, \forall \, i>1)$$
et 
$$H^1_{\fpqc}(\Spec(\ZZ), M\otimes\HH)\simeq M\otimes \hat{\ZZ}.$$
\end{cor}

\begin{proof}
On applique le corollaire \ref{c4}
à $X=K(M,n)$ et on utilise la formule
$$\pi_i(K(A,n)(\ZZ))\simeq H^{n-i}_{\fpqc}(\Spec(\ZZ),A)$$
pour tout faisceau abélien $A$.
\end{proof}

Pour terminer, nous revenons sur l'$\s$-foncteur
$X \mapsto (X\otimes \ZZ)^{\uni}(\ZZ)$ de l'$\s$-catégorie
$Top_{\geq 1}^{\tf}$ des espaces simplement connexes et de type fini
vers $Top$. 
Nous avons vu qu'il est équivalent à l'$\s$-foncteur
$$X \mapsto X_\QQ \times_{X_{\A}}\prod_p X_p^{S^1}.$$
Nous pouvons dire un peu plus. Tout d'abord, notons 
que le morphisme $\prod_p X_p^{S^1} \longrightarrow X_{\A}$ 
possède en réalité une factorisation canonique
$$\prod_p X_p^{S^1} \longrightarrow \prod_p X_p \longrightarrow X_{\A},$$
où le premier morphisme est l'évaluation au point de base de $S^1$.
Pour voir cela on introduit l'anneau 
$\hat{\bar{\ZZ}}:=\prod_p\hat{\bar{\ZZ}}_p$
où $\hat{\ZZ}_p \subset \hat{\bar{\ZZ}}_p$ est l'extension étale maximale non-ramifiée. 
En d'autres termes 
$\hat{\bar{\ZZ}}_p=\WW_{p^\s}(\bar{\FF}_p)$. Les Frobenius
fournissent un automorphisme $Fr$ de l'anneau  $\hat{\bar{\ZZ}}$. Nous
noterons aussi $\bar{\A}:=\hat{\bar{\ZZ}}\otimes \QQ$. Notons immédiatement que l'on
dispose d'une suite exacte courte
$$\xymatrix{
0 \ar[r] & \A \ar[r] & \bar{\A} \ar[r]^-{1-Fr} & \bar{\A} \ar[r] & 0.}$$

On considère maintenant le diagramme commutatif suivant
$$\xymatrix{
(X\otimes \ZZ)^{\uni}(\hat{\ZZ}) \ar[r] \ar[d] & (X\otimes \ZZ)^{\uni}(\hat{\bar{\ZZ}}) \ar[d] \\
(X\otimes \ZZ)^{\uni}(\A) \ar[r] & (X\otimes \ZZ)^{\uni}(\bar{\A}).}$$
Le Frobenius opère de manière compatible sur les membres
de droites, et nous noterons par $(-)^{Fr}$ les points fixes
homotopiques de cette action.

\begin{lem}\label{lfix}
Soit $R$ une $\QQ$-algèbre commutative et $Fr : R \longrightarrow R$
un endomorphisme de $R$. On suppose que $1-Fr : R \longrightarrow R$
est surjectif et on note $R^{Fr} \subset R$ le sous-anneau 
des points fixes de $F$ sur $R$. Alors, pour tout espace 
simplement connexe de type fini $X$, le morphisme naturel
$$(X\otimes \ZZ)(R^{Fr}) \longrightarrow (X\otimes \ZZ)(R)^{Fr}$$
est une équivalence faible.
\end{lem}

\begin{proof}
On connait les groupes d'homotopie
de $(X\otimes \ZZ)(R)$ et $(X\otimes \ZZ)(R^{Fr})$,
qui ne sont autre que les $\pi_i(X)\otimes R$ et $\pi_i(X)\otimes (R^{Fr})$. 
Le lemme se vérifie alors directement en considérant la suite
exacte longue en homotopie pour les points fixes homotopiques
de $Fr$
$$\xymatrix{\dots \ar[r] & 
\pi_i((X\otimes \ZZ)(R)^{Fr}) \ar[r] & \pi_i((X\otimes \ZZ)(R))  \ar[r]^{1-Fr} & 
\pi_i((X\otimes \ZZ)(R)) \ar[r] &  \pi_{i-1}((X\otimes \ZZ)(R)^{Fr})}$$
qui montre que $\pi_i((X\otimes \ZZ)(R)^{Fr})$ s'identifie canoniquement
à $\pi_i(X)\otimes (R^{Fr})$. 
\end{proof}

Le lemme précédent nous dit que le morphisme naturel
$$(X\otimes \ZZ)^{\uni}(\A) \longrightarrow (X\otimes \ZZ)^{\uni}(\bar{\A})^{Fr}$$
est une équivalence. Comme l'action
de Frobenius est canoniquement triviale sur $(X\otimes \ZZ)^{\uni}(\hat{\bar{\ZZ}})$, nous 
en déduisons un morphisme naturel $\phi$ qui fait commuter le diagramme ci-dessous
$$\xymatrix{
(X\otimes \ZZ)^{\uni}(\hat{\ZZ}) \ar[r] \ar[d] & (X\otimes \ZZ)^{\uni}(\hat{\bar{\ZZ}}) \ar[d]
\ar[ld]_-{\phi} \\
(X\otimes \ZZ)^{\uni}(\A) \ar[r] & (X\otimes \ZZ)^{\uni}(\bar{\A}).}$$
Ce morphisme $\phi$ fournit la factorisation canonique cherchée.

Nous pouvons donc maintenant écrire de manière fonctorielle en $X$ un 
diagramme commutatif à carrés cartésiens
$$\xymatrix{
& (X\otimes \ZZ)^{\uni}(\ZZ) \ar[r] \ar[d] & \prod_p X_p^{S^1} \ar[d] \\
X \ar[r]^-{\alpha} \ar[ru] & Y \ar[r] \ar[d] &  \prod_p X_p \ar[d] \\
 & X_\QQ \ar[r] & X_{\A}.}$$
Par inspection de la suite exacte longue en homotopie, on voit que le morphisme
$\alpha$ est en réalité une équivalence. Nous avons donc montré le corollaire suivant.

\begin{cor}\label{c6}
Pour tout $X\in Top_{\geq 1}^{\tf}$, le morphisme d'adjonction
$X \longrightarrow (X\otimes \ZZ)^{\uni}(\ZZ)$
possède une rétraction 
$$r : (X\otimes \ZZ)^{\uni}(\ZZ) \longrightarrow X,$$
fonctorielle en $X$.
\end{cor}

On déduit de ce corollaire le fait important suivant.

\begin{cor}\label{c7}
L'$\s$-foncteur de schématisation 
$$(-\otimes \ZZ)^{\uni} : Top_{\geq 1}^{\tf} \longrightarrow ChAff$$
possède les propriétés suivantes.
\begin{enumerate}
\item Il est injectif sur les classes d'équivalences
d'objets: deux espaces simplement connexes et de type fini
$X$ et $Y$ sont faiblement équivalents si et seulement 
si les champs $(X\otimes \ZZ)^{\uni}$ et $(Y\otimes \ZZ)^{\uni}$
sont équivalents.

\item Pour deux espaces $X$ et $Y$ simplement connexes et de type fini, 
le morphisme d'ensembles simpliciaux
$$\Map_{Top}(X,Y) \longrightarrow \Map_{St_\ZZ}((X\otimes \ZZ)^{\uni},(Y\otimes \ZZ)^{\uni})$$
possède une rétraction, et en particulier est injectif en homotopie.
\end{enumerate}
\end{cor}

\begin{proof}
Commençons par montrer $(2)$. 
Notons 
$i_X : X \longrightarrow (X\otimes \ZZ)^{\uni}(\ZZ)$ l'unité de l'adjonction, et
$r_X : (X\otimes \ZZ)^{\uni}(\ZZ) \longrightarrow X$
la rétraction du corollaire \ref{c6}. On définit
$$R : \Map_{St_\ZZ}((X\otimes \ZZ)^{\uni},(Y\otimes \ZZ)^{\uni})
\longrightarrow \Map_{Top}(X,Y)$$
par la formule
$R(f):=r_Y\circ f \circ i_X.$
Pour $(1)$, nous prétendons que 
la formule pour $R(f)$ ci-dessus préserve les morphismes induits en 
homotopie au sens suivant. Le morphisme $$f : (X\otimes \ZZ)^{\uni}(\ZZ) \longrightarrow
(Y\otimes \ZZ)^{\uni}(\ZZ)$$ 
induit d'après le théorème \ref{p3} un morphisme sur les faisceaux d'homotopie
$$\pi_*(f) : \pi_*(X)\otimes \HH \longrightarrow \pi_*(Y)\otimes \HH.$$
De même, $R(f) : X \longrightarrow Y$ induit 
une application en homotopie et donc à son tour un morphisme
$$\pi_*(R(f)) : \pi_*(X)\otimes \HH \longrightarrow \pi_*(Y)\otimes \HH.$$
Nous affirmons que les deux morphismes $\pi_*(f)$ et 
$\pi_*(R(f))$ sont égaux, ce qui se voit par construction 
de la rétraction $R$.

Ainsi, si $f$ est une équivalence, $\pi_*(f)$ est un isomorphisme. Or, le foncteur
qui à un groupe abélien de type fini $M$ associe le faisceau
abélien $M\otimes \HH$ est conservatif, comme cela se voit par exemple
en appliquant le foncteur des sections globales sur $\Spec(\ZZ)$ (voir le corollaire
\ref{c5}).
\end{proof}

\begin{rmk}
La discussion menant au corollaire \ref{c6} montre aussi que 
l'$\s$-foncteur $Top_{\geq 1}^{\tf} \longrightarrow Top$, qui à 
$X$ associe $(X\otimes \ZZ)^{\uni}(\ZZ)$, est équivalent 
à 
$$X \mapsto X\times_{\prod_p X_p}\prod_p X_p^{S^1}.$$
\end{rmk}

\begin{appendix}

\section{Le théorème d'Eilenberg-Moore}\label{appA}

Dans cet appendice nous rappelons le théorème d'Eilenberg-Moore
(voir \cite{smi}) et sa version légèrement améliorée en termes
d'algèbres cosimpliciales commutatives, dont nous redonnons une preuve.
Il sera utilisé de manière essentielle au cours de la preuve
du théorème \ref{p3} afin de montrer que l'affinisation 
d'espaces simplement connexes et de type fini commute avec 
les décompositions de Postnikov.

Nous rappelons qu'un espace $X$ est dit de \emph{type fini}, s'il est 
faiblement équivalent à un CW complexe ne possédant qu'un nombre
fini de cellules en chaque degré. Un espace
connexe et simplement connexe est de type fini si et seulement si tous ses groupes d'homotopie sont
de type fini.

\begin{thm}\label{ta1}
Soit 
$$\xymatrix{Z \ar[r] \ar[d] & Y \ar[d] \\
X \ar[r] & S}$$
un carré homotopiquement cartésien d'espaces topologiques. On suppose
de plus que les conditions suivantes sont satisfaites.
\begin{enumerate}
\item $S$ est un espace (connexe et) simplement connexe.

\item La fibre homotopique de $X \longrightarrow S$
est un espace de type fini.

\end{enumerate}
Alors, le morphisme naturel d'algèbres commutatives cosimpliciales
$$C^*(X,\ZZ)\stackrel{\mathbb{L}}{\otimes}_{C^*(S,\ZZ)}C^*(Y,\ZZ) \longrightarrow C^*(Z,\ZZ)$$
est un quasi-isomorphisme. 
\end{thm}

\begin{proof}
Représentons le diagramme d'espaces par un diagramme
de CW complexes tel que tous les morphismes soient 
des fibrations de Serre. 
On se place dans l'$\s$-catégorie $\D(X)$ des complexes
de faisceaux abéliens sur $X$ munie de sa structure symétrique monoïdale
usuelle. Notons
$$\pi : Z \longrightarrow S \qquad f : X \longrightarrow S \qquad g : Y \longrightarrow S$$
les projections. L'$\s$-foncteur d'images directes $\pi_*$ est lax monoïdal,
car adjoint à droite de $\pi^*$ qui est monoïdal, et l'on dispose ainsi d'un morphisme naturel
dans $\D(S)$
$$f_*(\ZZ)\stackrel{\mathbb{L}}{\otimes} g_*(\ZZ) \longrightarrow \pi_*(\ZZ).$$
Sur un ouvert contractile $U \subset S$, ce morphisme s'identifie en un morphisme
entre objets constants qui n'est autre que le morphisme de K\"{u}nneth
$$C^*(X_s,\ZZ) \stackrel{\mathbb{L}}{\otimes} C^*(Y_s,\ZZ) \longrightarrow C^*(Z_s,\ZZ),$$
où $X_s$, $Y_s$ et $Z_s$ sont les fibres de $X$, $Y$ et $Z$ au-dessus d'un point fixé $s\in U$.
D'après l'hypothèse $(2)$ il s'agit donc d'un quasi-isomorphisme, et ainsi le morphisme
naturel induit donc une équivalence dans $\D(S)$
$$f_*(\ZZ)\stackrel{\mathbb{L}}{\otimes} g_*(\ZZ) \simeq \pi_*(\ZZ).$$
Nous considérons maintenant l'$\s$-foncteur de sections globales sur $S$
$\Gamma : \D(S) \longrightarrow \D(*),$
qui est aussi un $\s$-foncteur lax monoïdal. En particulier, il induit un nouvel
$\s$-foncteur lax monoïdal
$$\Gamma : \D(S) \longrightarrow C^*(S,\ZZ)-Mod,$$
où $C^*(S,\ZZ)$ est ici considéré comme une $E_\s$-algèbre dans $\D(*)$
et $C^*(S,\ZZ)-Mod$ et son $\s$-catégorie monoïdale des modules dans $\D(*)$.
L'équivalence ci-dessus
fournit ainsi un morphisme naturel
$$\Gamma(f_*(\ZZ))\stackrel{\mathbb{L}}{\otimes}_{C^*(S,\ZZ)} \Gamma(g_*(\ZZ)) \longrightarrow
\Gamma(f_*(\ZZ) \stackrel{\mathbb{L}}{\otimes} g_*(\ZZ)) \simeq \Gamma(\pi_*(\ZZ)).$$
Ce morphisme dans $\D(*)$ s'identifie au morphisme de l'énoncé du théorème. Pour montrer 
qu'il s'agit d'une équivalence il suffit donc de montrer que le morphisme
$$\Gamma(f_*(\ZZ)) \stackrel{\mathbb{L}}{\otimes}_{C^*(S,\ZZ)} \Gamma(g_*(\ZZ)) \longrightarrow
\Gamma(f_*(\ZZ) \stackrel{\mathbb{L}}{\otimes} g_*(\ZZ))$$
est une équivalence. 

Pour cela, on considère l'énoncé général suivant: étant donnés
$E$ et $F$ des objets de $\D(S)$, le morphisme
$$\Gamma(E) \stackrel{\mathbb{L}}{\otimes}_{C^*(S,\ZZ)}\Gamma(F) \longrightarrow
\Gamma(E\stackrel{\mathbb{L}}{\otimes}F)$$
est une équivalence. Cet énoncé est trivialement satisfait pour $E$ 
le faisceaux constant $\ZZ$ (et tout $F$), 
car par construction $\Gamma(\ZZ) \simeq C^*(S,\ZZ)$
comme $C^*(S,\ZZ)$-modules. Par stabilité l'énoncé est vrai pour 
n'importe quels objets $E$ et $F$, avec $E$ appartenant à l'enveloppe
triangulée de l'objet $\ZZ$ dans $\D(S)$. Enfin, supposons que
$E$ soit un objet de $\D(S)$ vérifiant les hypothèses suivantes.
\begin{enumerate}

\item Les faisceaux de cohomologie $H^i(E)$ s'annulent pour $i<0$.

\item Les faisceaux de cohomologie $H^i(E)$ sont 
localement constants (et donc constants car $S$ est $1$-connexe) de fibre de type fini pour tout $i$.

\end{enumerate}

On peut alors écrire $E$ comme une colimite filtrante de ses tronqués 
$E \simeq \mathrm{colim}_n E_{\leq n}$. On a alors clairement
$\mathrm{colim}_n \Gamma(E_{\leq n}) \simeq \Gamma(E).$
De même, si $F\in \D(S)$ est un second objet satisfaisant la condition $(1)$ 
ci-dessus, alors on a 
$\mathrm{colim}_{n} (E_{\leq n}\stackrel{\mathbb{L}}{\otimes}F) \simeq \Gamma(E\stackrel{\mathbb{L}}{\otimes}F).$
Ainsi, pour deux tels $E$ et $F$  le morphisme
$$\Gamma(E) \stackrel{\mathbb{L}}{\otimes}_{C^*(S,\ZZ)} \Gamma(F) \longrightarrow
\Gamma(E\stackrel{\mathbb{L}}{\otimes}F)$$
est équivalent au morphisme
$$\mathrm{colim}_n (\Gamma(E_{\leq n}) 
\stackrel{\mathbb{L}}{\otimes}_{C^*(S,\ZZ)} \Gamma(F)) \longrightarrow
\mathrm{colim}_{n} \Gamma(E_{\leq n}\stackrel{\mathbb{L}}{\otimes}F),$$
qui est donc une équivalence car chacun des $E_{\leq n}$  
est dans l'enveloppe triangulée de $\ZZ$ dans $\D(S)$.

Cela démontre le théorème, en prenant $E=f_*\ZZ$, resp. $F=g_*\ZZ$,
qui 
satisfont bien aux conditions $(1)-(2)$, resp. à la condition $(1)$, d'après nos hypothèses
sur le morphisme $f$.
\end{proof}

\section{Le topos \fpqc}\label{appB}

Nous rappelons ici quelques faits sur le topos \fpqc, qui demandent
certains soins de nature ensemblistes. Il s'agit d'énoncés affirmant que
les limites le long de la catégorie $\NN^+$ (les entiers naturels
muni de l'ordre inverse de l'ordre naturel) se comportent essentiellement
comme dans le topos des ensembles. Ces résultats seront utilisés
de manière essentielle pour contrôler certaines limites
de champs (voir par exemple la preuve de la proposition \ref{p2}).

Nous notons
$\UU \in \VV$ deux univers de Grothendieck, et supposons
qu'il existe $x\in \UU$ qui est infini.

Soit $Aff$ la catégorie des schémas affines qui sont éléments de $\UU$. Par choix
des univers $Aff$ est une catégorie élément de $\VV$.
La catégorie
des (gros) préfaisceaux sur $Aff$ est par définition
$$\widehat{Aff}:=Fun(Aff^{\op},Ens_\VV),$$
la catégorie des foncteurs de $Aff^{\op}$ à valeurs dans la
catégorie des ensembles éléments de $\VV$. La catégorie
$Aff$ est munie d'une topologie de Grothendieck pour laquelle les familles
couvrantes $$\{\Spec(A_i) \rightarrow \Spec(A)\}_{i\in I}$$ dans $Aff$ ($I \in \UU$) 
sont celles
telles que $\Spec(A_i) \rightarrow \Spec(A)$ est plat pour tout $i$, 
et de plus il existe
un sous-ensemble fini $I'\subset I$ avec $\coprod_{i\in I'}\Spec(A_i) \rightarrow \Spec(A)$
surjectif. Cette topologie est la topologie $\fpqc$. On note
$Aff^{\sim,\fpqc} \subset \widehat{Aff}$ la sous-catégorie pleine 
formée des faisceaux \fpqc. C'est encore une catégorie $\VV$-petite, qui possède
tout type de limites et colimites $\UU$-petites.

On dispose d'une adjonction 
$$a : \widehat{Aff} \rightleftarrows Aff^{\sim,\fpqc} : j$$
où l'adjoint à droite $j$ est le foncteur d'inclusion, et 
l'adjoint à gauche $a$ est le foncteur de faisceautisation \fpqc. 

Une des propriétés remarquables du topos $Aff^{\sim,\fpqc}$ est la
suivante.

\begin{prop}\label{pa}
Le foncteur $a$ commute aux limites dénombrables.
\end{prop}

\begin{proof}
On sait que le foncteur $a$ est exact, et il suffit donc
de montrer qu'il commute aux produits dénombrables. Pour cela, soit 
$\{F_i\}_{i\geq 0}$ une famille dénombrable d'objets de $\widehat{Aff}$ 
et considérons le morphisme de faisceaux
$$\alpha : a(\prod_i F_i) \longrightarrow \prod_i a(F_i).$$
On commence par remarquer que ce morphisme est un épimorphisme. En effet, 
soit $X\in Aff$ et considérons $(x_i)_{i\geq 0}\in \prod_i a(F_i)(X)$ une section
de $\prod_i a(F_i)$ sur $X$. Pour tout $i\geq 0$, il existe 
un recouvrement \fpqc $u_i : Y_i \rightarrow X$ et un élément $y_i \in F_i(Y_i)$
tels que l'application canonique $F_i(Y_i) \rightarrow a(F_i)(Y_i)$
envoie $y_i$ sur $u_i^*(x_i)$. 

Notons $Y_i=\Spec(A_i)$, de telle sorte que $A_i$ soit une $A$-algèbre fidèlement 
plate. On considère la somme infinie dans la catégorie des $A$-algèbres commutatives $B:=\otimes_{i\geq 0}A_i$. 
Comme $B$ peut aussi s'écrire comme une colimite le long de $\mathbb{N}$,
du système $(n \mapsto \otimes_{0\leq i\leq n}A_i)$, on voit que $B$ est encore une 
$A$-algèbre fidèlement plate. On note $Z=\Spec(B)$, qui est aussi le produit fibré
de la famille des $Y_i \rightarrow X$ dans $Aff$, ainsi que
 $p_i : Z \rightarrow Y_i$ et $p : Z \rightarrow X$ les projections. 
Alors, on voit que l'élément $p^*(x_i)_{i\geq 0} \in \prod_i a(F_i(Z))$ est l'image de l'élément 
$(p_i^*(y_i))_{i\geq 0} \in \prod_{i}F_i(Z)$ par le morphisme naturel $\prod_iF_i \rightarrow \prod_ia(F_i)$. 
Ceci montre que le morphisme $\alpha$ est un épimorphisme. 

On procède de manière similaire pour voir que $\alpha$ est un monomorphisme de faisceaux. 
\end{proof}

La proposition \ref{pa} implique en particulier que les produits dénombrables
sont exacts dans la catégorie des faisceaux abéliens. De cela découle le corollaire 
suivant (voir aussi \cite{basch} pour la notion de topos \emph{replete}
dont le topos \fpqc est un exemple).

\begin{cor}\label{ca}
Soit 
$$E_* : \xymatrix{\dots \ar[r] & E_n \ar[r] & E_{n-1} \ar[r] & \dots \ar[r] & E_0}$$
une tour dénombrable
de préfaisceaux abéliens sur le site \fpqc des schémas affines. Alors, 
les morphismes naturels
$$a(\lim_n E_n) \longrightarrow \lim_n a(E_n) \qquad
a({\lim_n}^1 E_n) \longrightarrow {\lim_n}^1 a(E_n)$$
sont des isomorphismes de faisceaux abéliens $($où les membres
de droites sont calculés dans la catégorie des faisceaux
\fpqc abéliens et ceux de gauches dans la catégorie des préfaisceaux abéliens$)$. De plus
$\lim_n^ia(E_n)=0$ pour tout $i>1$.
\end{cor}

Les résultats précédents possèdent une version homotopique pour les champs.
On considère l'adjonction d'$\s$-catégories
$$a : Fun(Aff^{\op},Top_{\VV}) \rightleftarrows St_\ZZ : j,$$
où $j$ est l'inclusion et $a$ le foncteur champ associé, que l'on sait
être exact. Nous travaillons ici
avec des champs hypercomplets (voir \cite{HT}), c'est-à-dire que les équivalences
faibles se détectent sur les faisceaux d'homotopie. Ainsi, la proposition
\ref{pa} implique-t-elle que l'$\s$-foncteur $a$ commute aux
produits dénombrables, et donc à tout type de limites homotopiques
dénombrables. Un conséquence importante est le fait que l'$\s$-foncteur
$$Top_{\VV} \longrightarrow St_\ZZ$$
qui à un espace associe le champ constant correspondant, commute avec 
les limites homotopiques dénombrables. 

Le cas qui nous intéressera particulièrement est le corollaire suivant, 
qui est une nouvelle interprétation des résultats de  \cite{to}.

\begin{cor}\label{ca2}
Soit $X$ un espace simplement connexe et de type fini. Soit $X \longrightarrow X_p$
son complété p-adique et notons encore $X_p$ le champ
constant associé. Alors, pour tout corps de caractéristique $p>0$, 
il existe une équivalence canonique de champs \fpqc sur $\Spec(k)$
$$X_p \simeq (X \otimes k)^{\uni}.$$
\end{cor}

\begin{proof}
Comme $(X \otimes k)^{\uni}$
est obtenu par changement de base de $(X \otimes \FF_p)^{\uni}$
il suffit de traiter le cas $k=\FF_p$.
On rappelle que $X_p$ est caractérisé
par le fait que le morphisme induit en homotopie 
$\pi_i(X) \longrightarrow \pi_i(X_p)$ fait de $\pi_i(X_p)$ le
complété $p$-adique de $\pi_i(X)$. De plus, le morphisme 
$X \longrightarrow X_p$ induit une équivalence sur les $k$-algèbres
cosimpliciales de cohomologie
$$C^*(X_p,\FF_p) \simeq C^*(X,\FF_p).$$
Ceci nous donne, par adjonction, le morphisme canonique
$$X_p \longrightarrow \mathbb{R}\Spec(C^*(X_p,\FF_p)) \simeq
(X\otimes \FF_p)^{\uni}.$$
Sur les groupes d'homotopie, ce morphisme induit le morphisme canonique
$$\pi_i(X_p) \longrightarrow \lim_n (\pi_i(X)\otimes \ZZ/p^n),$$
où la limite de droite est prise dans la catégorie
des faisceaux \fpqc et le membre de gauche est le faisceau
constant de fibre le complété $p$-adique de $\pi_i(X)$. Le fait que
ce morphisme soit un isomorphisme de faisceaux se déduit de la proposition \ref{pa}.
\end{proof}

\end{appendix}


\ifx\undefined\bysame
\newcommand{\bysame}{\leavevmode\hbox to3em{\hrulefill}\,}
\fi

\end{document}